\documentclass[11pt,a4paper]{article}
\usepackage[T1]{fontenc}
\usepackage[utf8]{inputenc}
\usepackage{lmodern}
\usepackage[margin=1in]{geometry}
\usepackage{amsmath,amssymb,amsthm}
\usepackage{microtype}
\usepackage{tikz}
\usetikzlibrary{arrows.meta,positioning}
\usepackage[hidelinks]{hyperref}
\hypersetup{
 pdftitle={Small and large quasi-kernels: a weighted transfer theorem},
 pdfauthor={Gregory Gutin, Xiangzhou Liu and Xianhao Meng},
 pdfsubject={Weighted quasi-kernel transfer, universal equivalence, and break digraphs},
 pdfkeywords={quasi-kernel, small quasi-kernel, large quasi-kernel, weighted neighbourhood, matching, break digraph}
}
\newtheorem{theorem}{Theorem}[section]
\newtheorem{lemma}[theorem]{Lemma}
\newtheorem{corollary}[theorem]{Corollary}
\newtheorem{conjecture}[theorem]{Conjecture}
\newcommand{\gq}{\gamma_2}
\newcommand{\Znonneg}{\mathbb{Z}_{\ge 0}}

\title{Small and large quasi-kernels:\protect\\a weighted transfer theorem}
\author{Jiangdong Ai$^1$ \quad Gregory Gutin$^2$ \quad Xiangzhou Liu$^1$ \quad Xianhao Meng$^1$\\[4pt]
\small $^1$School of Mathematical Sciences and LPMC, Nankai University, China\\
\small $^2$Department of Computing, Security and Mathematics, Royal Holloway University of London, UK}
\begin{document}
\maketitle

\begin{abstract}
A quasi-kernel of a digraph $D$ is an independent set from which every
vertex is reachable by a directed path of length at most two. The Small
Quasi-kernel Conjecture asserts that every source-free digraph on $n$
vertices has a quasi-kernel of size at most $n/2$, whereas the Large
Quasi-kernel Conjecture asks for a quasi-kernel whose closed
out-neighbourhood has size at least $n/2$. Spiro proved that the former
implies the latter. In this paper, we establish a weighted transfer
theorem: if every nonnegative integer weighting of a source-free digraph
$D$ admits a quasi-kernel whose closed out-neighbourhood contains at
least a $\theta$-fraction of the total weight, then $D$ has a
quasi-kernel of size at most
$\lfloor(1-\theta)\alpha(D)+\theta\mu\rfloor$ for $0<\theta\le1/2$, where
$\mu$ is the maximum matching size between a maximum independent set
and its complement.
Using the independent blow-up construction of Spiro, we deduce that the
universal small and large quasi-kernel assertions are equivalent for
every such $\theta$, with no loss in the parameter. In particular, the
two conjectures are equivalent. As an application, we prove the
half-order bound for source-free break digraphs.
\end{abstract}

\section{Introduction}\label{sec:intro}

Throughout the paper, all digraphs are finite without loops or parallel arcs;
opposite arcs are allowed. A digraph with no pair of opposite arcs is
called an \emph{oriented graph}. For $X\subseteq V(D)$, let
\[
 N_D^+(X)=\{y\in V(D)\setminus X: x\to y\text{ for some }x\in X\},
 \qquad N_D^+[X]=X\cup N_D^+(X).
\]
We write $N_D^{\le2}[X]$ for the vertices reachable from $X$ by a
directed path of length at most two. A set $X\subseteq V(D)$ is \emph{independent}
if there is no arc between any pair of vertices in $X.$ 
An independent set $Q$ is a \emph{quasi-kernel}\footnote{We follow the convention that paths start in the quasi-kernel, not end there as in some other papers. Results stated with the reverse convention apply
after reversing all arcs.}  
of $D$ if $N_D^{\le2}[Q]=V(D)$. 
Chv\'atal and Lov\'asz~\cite{CL} proved that
every digraph has a quasi-kernel.

A vertex of indegree zero is called a \emph{source}; a digraph is
\emph{source-free} if it contains no source. Erd\H{o}s and
Sz\'ekely proposed the following conjecture; see e.g. the survey~\cite{Survey}.
\begin{conjecture}[Small Quasi-kernel Conjecture]\label{conj:small}
Every source-free digraph $D$ has a quasi-kernel $Q$ such that
$|Q|\le |V(D)|/2$.
\end{conjecture}
The source-free assumption is necessary, since every source belongs to
every quasi-kernel. The constant $1/2$ is best possible: consider
the directed 4-cycle. One approach to settle Conjecture \ref{conj:small} 
is to find two disjoint quasi-kernels, although this stronger property does not hold for all
source-free digraphs~\cite{GKTY}. It does hold for several special
classes, including semicomplete multipartite, quasi-transitive and
locally semicomplete digraphs~\cite{HH}. Other results establish the
half-order bound for source-free digraphs whose underlying graphs are
4-colourable~\cite{KLS}, for source-free digraphs admitting a
kernel~\cite{Hulst}, and for certain classes of split
digraphs~\cite{AGGYZ}. For further developments we refer to~\cite{Erdos,Survey}.

Spiro~\cite{Spiro} introduced a complementary way to measure the size
of a quasi-kernel, namely by the size of its closed out-neighbourhood.
\begin{conjecture}[Large Quasi-kernel Conjecture]\label{conj:large}
Every digraph $D$ has a quasi-kernel $Q$ such that
$|N_D^+[Q]|\ge |V(D)|/2$.
\end{conjecture}
Notice that Conjecture~\ref{conj:large} has no source-free assumption.
Spiro~\cite[Proposition~4.1]{Spiro} proved that the Small Quasi-kernel
Conjecture implies the Large Quasi-kernel Conjecture. Ai, Liu and
Peng~\cite{ALP} obtained a converse with a weaker constant and asked
whether the loss could be avoided. More precisely, for
$0<\theta\le1/2$, they asked whether the existence in \emph{every}
digraph on $n$ vertices of a quasi-kernel $Q$ with
$|N_D^+[Q]|\ge\theta n$ implies that \emph{every} source-free digraph
on $n$ vertices has a quasi-kernel of size at most $(1-\theta)n$.
The distinction between a universal assertion and a property of an
individual digraph is important here.

We obtain a positive answer by relating the weighted neighbourhood of
a quasi-kernel to its cardinality. For a nonnegative integer weighting
$f:V(D)\to\mathbb Z_{\ge0}$, write
$f(X)=\sum_{x\in X}f(x)$. Given $0<\theta\le1/2$, we say that $D$
satisfies \textup{($W_\theta$)} if every such weighting admits a
quasi-kernel $K$ with
\begin{equation}\label{eq:weighted-property}
 f\bigl(N_D^+[K]\bigr)\ge\theta f(V(D)).
\end{equation}
Let $\alpha(D)$ be the independence number of $D$, and let
$\gq(D)$ denote the minimum size of a quasi-kernel of $D$.
Our main result is the following.

\begin{theorem}\label{thm:transfer}
Let $0<\theta\le1/2$, and let $D$ be a source-free digraph
satisfying \textup{($W_\theta$)}. Choose a maximum independent set
$I^*$ of $D$, and let $\mu$ be the maximum size of a matching between
$I^*$ and $V(D)\setminus I^*$ in the underlying graph. Then
\begin{equation}\label{eq:transfer-main}
 \gq(D)\le
 \big\lfloor(1-\theta)\alpha(D)+\theta\mu\big\rfloor
 \le\big\lfloor(1-\theta)|V(D)|\big\rfloor.
\end{equation}
\end{theorem}

For $\theta=1/2$ it gives
\begin{equation}\label{eq:half-matching}
 \gq(D)\le\left\lfloor\frac{\alpha(D)+\mu}{2}\right\rfloor
 =\left\lfloor\frac{|V(D)|-d}{2}\right\rfloor, \text{  where  }
 d=|V(D)\setminus I^*|-\mu.
\end{equation}

Theorem~\ref{thm:transfer}, together with a weighted blow-up
observation of Spiro~\cite[Appendix~A.1 (arXiv version)]{Spiro},
yields the following equivalence. The reverse implication in the
proof retains the parameter $\theta$ exactly.
\begin{corollary}\label{cor:equivalence}
For every fixed $0<\theta\le1/2$, the following statements are
equivalent:
\begin{enumerate}
\item Every source-free digraph $D$ has a quasi-kernel of size at most
$(1-\theta)|V(D)|$.
\item Every digraph $D$ has a quasi-kernel $Q$ such that
$|N_D^+[Q]|\ge\theta|V(D)|$.
\end{enumerate}
In particular, the Small and Large Quasi-kernel Conjectures are
equivalent.
\end{corollary}

A \emph{break digraph}~\cite{Break} is an oriented graph $D$ admitting a
partition $V(D)=S\mathbin{\dot\cup}T$ such that $D[S]$ is acyclic
and $D[T]$ is a tournament, with no restriction on the arcs between
$S$ and $T$. In particular, $S$ need not be independent.
Ai, Huang, Liu, Tang and Xu~\cite{Break} proved a $2n/3$ bound for
source-free break digraphs. Nguyen, Scott and Seymour~\cite{NSS} established a weighted neighbourhood theorem for this
class. Combining it with Theorem~\ref{thm:transfer}, we obtain the following corollary.
\begin{corollary}\label{cor:break}
Every source-free break digraph $D$ has a quasi-kernel of size at
most $\lfloor |V(D)|/2\rfloor$.
\end{corollary}

The proof of Theorem~\ref{thm:transfer} proceeds in two steps. We first refine the neighbourhood-preserving selection argument of~\cite{KLS,ALP} to retain prescribed endpoints of matching arcs. We then apply an alternating-path argument, based on Berge's augmenting-path criterion~\cite{Berge}, to obtain a matching configuration suitable for the refined selection lemma. The details are given in Sections~\ref{sec:selection} and~\ref{sec:mixed}, and the two corollaries are derived in Section~\ref{sec:applications}.

\section{Neighbourhood-preserving selection}\label{sec:selection}

Two vertex sets are \emph{anticomplete} if no arc joins a vertex of
one set to a vertex of the other. We first show that a quasi-kernel can be reduced while preserving
its open out-neighbourhood and certain prescribed vertices.

\begin{lemma}\label{lem:selection}
Suppose that a loopless digraph $D$ has a partition
$V(D)=H\mathbin{\dot\cup}C\mathbin{\dot\cup}I$, where $I$ is independent
and anticomplete to $H$. Let $K$ be a quasi-kernel of $D$. Suppose that
$X\subseteq K\cap I$ admits an injection $\phi:X\to C$ such that
$x\to\phi(x)$ is an arc for every $x\in X$. There is a set $R\subseteq K$
containing $(K\setminus I)\cup X$ such that
\begin{equation}\label{eq:preservation}
 N_D^+(R)=N_D^+(K),\qquad
 C\cap N_D^+[R]=C\cap N_D^+[K],
\end{equation}
and
\begin{equation}\label{eq:representative-cost}
 |R|\le |K\cap H|+|C\cap N_D^+[K]|.
\end{equation}
If $Z=V(D)\setminus N_D^{\le2}[R]$, then $Z\subseteq K\cap I$ and
$P=R\cup Z\subseteq K$ is a quasi-kernel of $D$.
\end{lemma}

\begin{proof}
Let $F=(K\setminus I)\cup X$, and choose an inclusion-minimal set
$J\subseteq(K\cap I)\setminus X$ such that, with $R=F\cup J$,
$N_D^+(R)=N_D^+(K)$. Such a set exists since $F\cup((K\cap I)\setminus
X)=K$.

For each $j\in J$, the minimality of $J$ gives a vertex
$c_j\in N_D^+(j)\setminus N_D^+(R\setminus\{j\})$. Since $I$ is
independent and anticomplete to $H$, we have $c_j\in C$. Moreover,
these vertices are pairwise distinct and lie outside $N_D^+(F)$.
The sets $K\cap C$, $\phi(X)$ and $\{c_j:j\in J\}$ are pairwise
disjoint: this follows from the independence of $K$, the injectivity
of $\phi$, and the choice of the private out-neighbours $c_j$.
They are all contained in $C\cap N_D^+[K]$. Hence
\[
 |K\cap C|+|X|+|J|\le |C\cap N_D^+[K]|,
\]
which, upon adding $|K\cap H|$, gives \eqref{eq:representative-cost}.
As $R$ retains $K\cap C$ and $N_D^+(R)=N_D^+(K)$, both equalities in
\eqref{eq:preservation} follow.

We claim that every vertex outside $K$ is still reachable from $R$
in at most two steps. Indeed, a vertex in $N_D^+(K)$ is also in
$N_D^+(R)$. If $k\to u\to y$ is a path with $k\in K$, then
$u\in N_D^+(K)=N_D^+(R)$, and there exists $r\in R$ with
$r\to u\to y$. Thus $Z\subseteq K$. Since $K\setminus I\subseteq R$,
we have $Z\subseteq K\cap I$. Consequently $P=R\cup Z$ is an
independent set, as $P\subseteq K$, and is a quasi-kernel of $D$.
\end{proof}

\section{Matching transfer and alternating paths}\label{sec:mixed}
\subsection{A mixed-direction matching lemma}

In the next lemma the matching is taken in the underlying graph;
its arcs need not have the same direction in $D$. If a matched pair
has two opposite arcs, we assign it to the first of the two classes
below.


\begin{lemma}\label{lem:mixed}
Fix $0<\theta\le 1/2$. Let $D$ be a loopless digraph with property
~\textup{($W_\theta$)} and a partition
$V(D)=H\mathbin{\dot\cup}C\mathbin{\dot\cup}I$, where $I$ is independent
and anticomplete to $H$. Suppose that a matching between $C$ and $I$ in
the underlying graph saturates $C$, and that each unmatched vertex of
$I$ has an in-neighbour in $D$. Let $k=|C|$ and $i=|I|$. Then there is a
quasi-kernel $P$ of $D$ such that
\begin{equation}\label{eq:matched-bound}
 |P|\le (1-\theta)i+\theta k+\alpha(D[H]).
\end{equation}
\end{lemma}

\begin{proof}
For $c\in C$, denote its matching mate by $m(c)\in I$, and let
\[
 C_+=\{c\in C:c\to m(c)\},\qquad C_-=C\setminus C_+,
 \qquad I_+=m(C_+),\quad I_-=m(C_-).
\]
Every $c\in C_-$ satisfies $m(c)\to c$. Put $b=|C_-|=|I_-|$, and
let $J=I\setminus(I_+\cup I_-)$ be the set of unmatched vertices, so
$u:=|J|=i-k$.

For each $x\in J$, choose an in-neighbour $p(x)\in C$; such a vertex
exists by the hypotheses. Define a nonnegative integer weighting $f$ by
\[
 f(c)=|\{x\in J:p(x)=c\}|\quad(c\in C),
 \qquad f(y)=0\quad(y\notin C).
\]
Note that $f(V(D))=u$. By \textup{($W_\theta$)}, there is a
quasi-kernel $K$ of $D$ such that
\begin{equation}\label{eq:weighted-witness}
 t:=f\bigl(N_D^+[K]\bigr)\ge \theta u.
\end{equation}
For $u=0$ we use the zero weighting.

Apply Lemma~\ref{lem:selection} with $X=K\cap I_-$, mapping each
vertex of $X$ to its matching mate in $C_-$. We obtain sets $R$ and
$Z$ such that $P=R\cup Z$ is a quasi-kernel and
\eqref{eq:preservation} holds. Write
\[
 c_+=|C_+\cap N_D^+[K]|,\qquad
 c_-=|C_-\cap N_D^+[K]|.
\]
We now count the vertices of $I$ reached from $R$ in at most two
steps. Every vertex of $I_-$ is reached: the vertices in $K\cap I_-$
are retained, and those outside $K$ remain covered by
Lemma~\ref{lem:selection}. Also, for every
$c\in C_+\cap N_D^+[K]$, \eqref{eq:preservation} gives
$c\in N_D^+[R]$; the arc $c\to m(c)$ then shows that $m(c)$ is
reached within two steps. Finally, the weight $t$ counts exactly the
vertices $x\in J$ for which $p(x)\in N_D^+[K]$. Each is reached
through its assigned arc $p(x)\to x$. Although several vertices of
$J$ may have the same parent, they contribute separately to $t$.
These three groups are disjoint and have sizes $b$, $c_+$ and $t$,
respectively. Since $Z\subseteq K\cap I$, it follows that
\begin{equation}\label{eq:uncovered}
 |Z|\le i-b-c_+-t.
\end{equation}
On the other hand, \eqref{eq:representative-cost} gives
$|R|\le |K\cap H|+c_++c_-$. As $R$ and $Z$ are disjoint, this and
\eqref{eq:uncovered} yield
\begin{align*}
 |P|&\le |K\cap H|+c_++c_-+i-b-c_+-t\\
     &=|K\cap H|+i-(b-c_-)-t\\
     &\le |K\cap H|+i-\theta u\\
     &=|K\cap H|+(1-\theta)i+\theta k.
\end{align*}
Here we used $c_-\le b$ and $t\ge\theta u$.
Finally, $|K\cap H|\le\alpha(D[H])$, proving
\eqref{eq:matched-bound}.
\end{proof}

\subsection{The alternating-path decomposition}

We next obtain a partition to which Lemma~\ref{lem:mixed} can be
applied. The construction uses alternating paths of a maximum
matching in the underlying graph.

\begin{lemma}\label{lem:decomposition}
Let $G$ be a simple undirected graph, let $I^*$ be a maximum
independent set, and put $C^*=V(G)\setminus I^*$. Let $M$ be a maximum
matching in the bipartite graph of all edges between $C^*$ and $I^*$,
with $\mu=|M|$ and $d=|C^*|-\mu$. There is a partition
$V(G)=H\mathbin{\dot\cup}C\mathbin{\dot\cup}I$ and a nonnegative integer
$v$ such that $I\subseteq I^*$ has no neighbours in $H$, the restriction of
$M$ to $C,I$ saturates $C$, and
\begin{equation}\label{eq:decomposition}
 \begin{gathered}
 |H|=2v+d,\qquad \alpha(G[H])=v,\\
 |C|+v=\mu,\qquad |I|+v=\alpha(G).
 \end{gathered}
\end{equation}
\end{lemma}

\begin{proof}
Let $L\subseteq C^*$ be the vertices unmatched by $M$, so $|L|=d$.
Orient the edges of the bipartite graph temporarily: nonmatching edges
point from $C^*$ to $I^*$, and matching edges point from $I^*$ to $C^*$.
Let $A\subseteq C^*$ and $B\subseteq I^*$ be the vertices reachable from
$L$ in this auxiliary orientation, allowing paths of length zero.

No vertex of $B$ is unmatched. Otherwise, a shortest auxiliary path
from $L$ to an unmatched vertex of $I^*$ has distinct vertices and
alternates between nonmatching and matching edges, starting and ending
with a nonmatching edge. Replacing its matching edges by its
nonmatching edges gives a matching larger than $M$, a contradiction.
Every vertex of $B$ therefore has its matching mate in $A$. Conversely,
a vertex of $A\setminus L$ can only be reached through its matching
edge, so its mate is in $B$. Consequently $M$ gives a bijection between
$B$ and $A\setminus L$. Put $v=|B|$; then $|A|=v+d$.

Every neighbour in $I^*$ of a vertex of $A$ belongs to $B$. A
nonmatching edge extends an auxiliary path from $L$. If the edge is a
matching edge, its $I^*$ endpoint is already in $B$ by the preceding
bijection. Thus $A$ is anticomplete to $I^*\setminus B$.

Set
\[
 H=A\cup B,\qquad C=C^*\setminus A,\qquad I=I^*\setminus B.
\]
The set $I$ is anticomplete to $A$, and to $B$ by independence of $I^*$;
hence it is anticomplete to all of $H$. Every unmatched vertex of $C^*$
lies in $A$, and every mate of a vertex of $B$ lies in $A$ as well.
It follows that every vertex of $C$ is matched to a vertex of $I$.
We have $|H|=2v+d$, and exactly $v$ edges of $M$ lie between $A$ and
$B$. The remaining edges match $C$ into $I$, giving $|C|+v=\mu$.
Also, $|I|+v=|I^*|=\alpha(G)$.

Finally, $B$ is an independent set of $G[H]$, so
$\alpha(G[H])\ge v$. If $G[H]$ had an independent set larger than $v$,
its union with the anticomplete set $I$ would be an independent set
larger than $I^*=B\cup I$. This contradicts the choice of $I^*$ and
proves $\alpha(G[H])=v$, completing \eqref{eq:decomposition}.
\end{proof}

The orientation used here is auxiliary; it is independent of the
arc directions in any digraph whose underlying graph is $G$.

\subsection{Proof of the parameter transfer theorem}
\begin{proof}[Proof of Theorem~\ref{thm:transfer}]
Let $n=|V(D)|$ and let $G$ be the underlying graph of $D$. Choose a maximum independent
set $I^*$ and a maximum matching $M$ between $I^*$ and
$C^*=V(G)\setminus I^*$, so $\mu=|M|$. Apply
Lemma~\ref{lem:decomposition} to obtain $H,C,I$ and $v$.

The set $I$ is independent and anticomplete to $H$, while the remaining
underlying matching saturates $C$. Since $D$ is source-free, every
unmatched vertex of $I$ has an in-neighbour. Such an in-neighbour must
lie in $C$, so Lemma~\ref{lem:mixed} applies to $D$ itself. With the
identities of Lemma~\ref{lem:decomposition}, it gives a quasi-kernel
$P$ satisfying
\begin{align*}
 |P|&\le \alpha(D[H])+(1-\theta)|I|+\theta|C|\\
     &=v+(1-\theta)(\alpha(D)-v)+\theta(\mu-v)\\
     &=(1-\theta)\alpha(D)+\theta\mu.
\end{align*}
As $\mu\le |C^*|=n-\alpha(D)$ and $0<\theta\le 1/2$,
\begin{align}\label{eq:parameter-order}
 \gq(D)&\le(1-\theta)\alpha(D)+\theta\mu\notag\\
 &\le\theta n+(1-2\theta)\alpha(D)\le(1-\theta)n.
\end{align}
Taking integer parts establishes~\eqref{eq:transfer-main}.
\end{proof}

The argument includes the case $u=0$, when the weight function in
Lemma~\ref{lem:mixed} is identically zero. If $C=\varnothing$, the
source-free assumption forces $I=\varnothing$, and the same conclusion
holds. The empty digraph has the empty quasi-kernel.

\section{Universal equivalence and break digraphs}\label{sec:applications}
\subsection{From universal unweighted neighbourhoods to weights}

We use the independent blow-up construction recorded by
Spiro~\cite[Appendix~A.1 (arXiv version)]{Spiro} to pass from a
universal unweighted assertion to the weighted property. The extra
unit in each blow-up size allows zero weights.

\begin{lemma}\label{lem:blowup}
Fix $0<\theta\le1/2$. Suppose that every loopless digraph $F$
has a quasi-kernel $Q$ with
$|N_F^+[Q]|\ge\theta|V(F)|$. Then every loopless digraph $D$
has property~\textup{($W_\theta$)}.
\end{lemma}

\begin{proof}
Fix $D$, a weight function $f:V(D)\to\Znonneg$, and a positive
integer $M$. For every $v\in V(D)$, replace $v$ by an independent set
$B_v$ of size $Mf(v)+1$. Replace each original arc $v\to w$ by all arcs
from $B_v$ to $B_w$. Let $D_M$ be the resulting loopless digraph.
Every $B_v$ is nonempty, even when $f(v)=0$.

By the hypothesis applied to $D_M$, there is a quasi-kernel $Q_M$
with $|N_{D_M}^+[Q_M]|\ge\theta|V(D_M)|$. Put
\[
 K_M=\{v\in V(D):Q_M\cap B_v\ne\varnothing\}.
\]
An arc between two vertices of $K_M$ would induce an arc between
the corresponding selected copies in $Q_M$, so $K_M$ is independent.
For every $v\in V(D)$, a path of length at most two from $Q_M$ to
any vertex of $B_v$ projects to a directed walk of length at most
two from $K_M$ to $v$. Such a walk can be shortened to a path.
Consequently, $K_M$ is a quasi-kernel of $D$.

Write $A_M=N_D^+[K_M]$ and $n=|V(D)|$. Every vertex of
$N_{D_M}^+[Q_M]$ lies in $\bigcup_{v\in A_M}B_v$. Therefore
\[
 Mf(A_M)+|A_M|=\sum_{v\in A_M}|B_v|
   \ge |N_{D_M}^+[Q_M]|
   \ge\theta\bigl(Mf(V(D))+n\bigr).
\]
As $D$ has only finitely many subsets, there is a quasi-kernel $K$
such that $K_M=K$ for arbitrarily large values of $M$.
Dividing by $M$ and letting $M\to\infty$ along this sequence gives
$f(N_D^+[K])\ge\theta f(V(D))$, as required.
\end{proof}

\subsection{Exact-parameter equivalence}
\begin{proof}[Proof of Corollary~\ref{cor:equivalence}]
The implication (1)$\Rightarrow$(2) is due to
Spiro~\cite[Proposition~4.1]{Spiro}; see also
Ai, Liu and Peng~\cite[Proposition~4]{ALP}, with all arcs reversed
to match our convention.

Suppose now that (2) holds. By Lemma~\ref{lem:blowup}, every digraph
satisfies \textup{($W_\theta$)}. Applying Theorem~\ref{thm:transfer}
to each nonempty source-free digraph $D$ gives
$\gq(D)\le(1-\theta)|V(D)|$. The assertion is trivial for the empty
digraph, so (1) follows.
\end{proof}

This proves the converse with no loss in the parameter, answering
the question of Ai, Liu and Peng~\cite{ALP}. Note that the corollary
concerns universal assertions. The proof uses the unweighted neighbourhood assertion for every
auxiliary blow-up $D_M$, not merely for the original digraph $D$.

\subsection{Application to break digraphs}\label{subsec:break}
We apply the following theorem of Nguyen, Scott and
Seymour~\cite{NSS}, which allows nonnegative integer weights.

\begin{theorem}[Theorem 3.2 in \cite{NSS}]
\label{thm:nss}
Let $D$ be an oriented graph with a partition
$V(D)=S\mathbin{\dot\cup}T$ such that $D[S]$ is acyclic and $D[T]$
is a tournament. Then, for every $f:V(D)\to\Znonneg$, there is a
quasi-kernel $K$ satisfying
\[
 f\bigl(N_D^+[K]\bigr)\ge\tfrac12 f(V(D)).
\]
\end{theorem}

\vspace{1mm}

\begin{proof}[Proof of Corollary~\ref{cor:break}]
Theorem~\ref{thm:nss} says precisely that every break digraph has
property~\textup{($W_{1/2}$)}. For a nonempty source-free break
digraph $D$, Theorem~\ref{thm:transfer} with $\theta=1/2$
yields a quasi-kernel of size at most $\lfloor|V(D)|/2\rfloor$.
\end{proof}

The bound is sharp. Indeed, the directed 4-cycle
$\overrightarrow{C_4}$ is a source-free break digraph and $\gamma_2(\vec C_4)=2$.




\begin{thebibliography}{99}
\bibitem{AGGYZ}
J.~Ai, S.~Gerke, G.~Gutin, A.~Yeo and Y.~Zhou,
Results on the small quasi-kernel conjecture,
\emph{Discrete Math.} \textbf{346}(7) (2023), 113435.

\bibitem{Break}
J.~Ai, T.~Huang, X.~Liu, C.~Tang and Z.~Xu,
\emph{Quasi-kernels in Hereditary Classes and Applications to Break Digraphs},
\href{https://arxiv.org/abs/2606.16571v2}{arXiv:2606.16571v2} (2026).

\bibitem{ALP}
J.~Ai, X.~Liu and F.~Peng,
A variable version of the quasi-kernel conjecture,
\emph{Electron. J. Combin.} \textbf{32}(2) (2025), P2.53.

\bibitem{Berge}
C.~Berge, Two theorems in graph theory,
\emph{Proc. Natl. Acad. Sci. USA} \textbf{43}(9) (1957), 842--844.

\bibitem{CL}
V.~Chv\'atal and L.~Lov\'asz,
Every directed graph has a semi-kernel,
in \emph{Hypergraph Seminar}, Lecture Notes in Mathematics 411,
Springer, 1974, p.~175.

\bibitem{Erdos}
P.~L.~Erd\H{o}s, Two conjectures on quasi-kernels,
in \emph{Fete of Combinatorics and Computer Science},
Bolyai Society Mathematical Studies 20, Springer, 2010,
pp.~357--358 (Open Problems, no.~4).

\bibitem{Survey}
P.~L.~Erd\H{o}s, E.~Gy\H{o}ri, T.~R.~Mezei, N.~Salia and M.~Tyomkyn,
On the small quasi-kernel conjecture---a survey,
in \emph{Information Theory and Related Fields},
Lecture Notes in Computer Science 14620, Springer, 2025,
pp.~98--110.

\bibitem{GKTY}
G.~Gutin, K.~M.~Koh, E.~G.~Tay and A.~Yeo,
On the number of quasi-kernels in digraphs,
\emph{J. Graph Theory} \textbf{46}(1) (2004), 48--56.

\bibitem{HH}
S.~Heard and J.~Huang,
Disjoint quasi-kernels in digraphs,
\emph{J. Graph Theory} \textbf{58}(3) (2008), 251--260.

\bibitem{KLS}
A.~V.~Kostochka, R.~Luo and S.~Shan,
Towards the small quasi-kernel conjecture,
\emph{Electron. J. Combin.} \textbf{29}(3) (2022), P3.49.

\bibitem{LMRVZ}
H.~Langlois, F.~Meunier, R.~Rizzi, S.~Vialette and Y.~Zhou,
Quasi-kernels in split graphs,
\emph{Discrete Appl. Math.} \textbf{361} (2025), 236--243.

\bibitem{NSS}
T.~Nguyen, A.~Scott and P.~Seymour,
Distant digraph domination,
\emph{Electron. J. Combin.} \textbf{33}(1) (2026), P1.32.
\href{https://doi.org/10.37236/13556}{doi:10.37236/13556}.

\bibitem{Spiro}
S.~Spiro, Generalized quasikernels in digraphs,
\emph{European J. Combin.} \textbf{133} (2026), 104307.
For the weighted/unweighted construction, see Appendix~A.1 of
\href{https://arxiv.org/abs/2404.07305v2}{arXiv:2404.07305v2}.

\bibitem{Hulst}
A.~van Hulst, \emph{Kernels and Small Quasi-Kernels in Digraphs},
\href{https://arxiv.org/abs/2110.00789}{arXiv:2110.00789} (2021).

\end{thebibliography}
\end{document}